\documentclass[12pt,twoside,reqno,psamsfonts]{amsart}

\usepackage[OT1]{fontenc}
\usepackage{type1cm}
\usepackage{enumerate}
\usepackage{amssymb}
\usepackage{mathscinet}
\usepackage[dvips]{graphicx}
\usepackage{geometry}        
\usepackage{version}         

\numberwithin{equation}{section}

\theoremstyle{plain}
\newtheorem{theorem}{Theorem}[section]

\newtheorem{lemma}[theorem]{Lemma}

\theoremstyle{remark}
\newtheorem{remark}[theorem]{Remark}

\newtheorem{example}[theorem]{Example}

\theoremstyle{definition}

\newcommand{\R}{\mathbb{R}}

\newcommand{\Z}{\mathbb{Z}}
\newcommand{\N}{\mathbb{N}}

\newcommand{\iii}{\mathtt{i}}

\DeclareMathOperator{\spt}{spt}

\DeclareMathOperator{\diam}{diam}

\begin{document}

\title
{Intermediate value property for the Assouad dimension of measures}

\author{Ville Suomala}
\address{Department of Mathematics and Statistics \\
        P.O.Box 8000\\
        FI-90014 University of Oulu\\
    Finland}
\email{ville.suomala@oulu.fi}

\thanks{The author was partially supported by the Academy of Finland CoE in Analysis and Dynamics research. 
}

\begin{abstract}
Hare, Mendivil, and Zuberman have recently shown that if $X\subset\R$ is compact and of non-zero Assouad dimension $\dim_A X$, then for all $s>\dim_A X$, $X$ supports measures with Assouad dimension $s$. We generalize this result to arbitrary complete metric spaces. 
\end{abstract}

\maketitle

\section{Introduction and notation}

The Assouad dimension has its origins in metric geometry and embedding problems, see e.g. \cite{Luukkainen1998, Heinonen2001}. More recently, it has been investigated thoroughly in fractal geometry. For a few of the many recent advances, see \cite{M2011,F2014,KL2017,FO2017,DFSU2019,FFS2020}.
The \emph{Assouad dimension} of a metric space $X=(X,d)$, denoted $\dim_A X$, is defined to be the infimum of those $\alpha\ge  0$ such that the following uniform bound holds for all $x\in X$, $0<r<R<\infty$: 
 \begin{equation}\label{eq:Assouad_set_def}
 N(x,R,r)\le C\left(\frac{R}{r}\right)^\alpha\,\,.
 \end{equation}
Here $N(x,R,r)$ denotes the minimal number of closed $r$-balls needed to cover the closed $R$-ball $B(x,R)=\{y\in X\,:\,d(x,y)\le R\}$ and $0<C<\infty$ is a constant independent of $x$, $r$, and $R$. If there is no $\alpha\in[0,\infty[$ fulfilling the requirement \eqref{eq:Assouad_set_def}, we assign $\dim_A X$ the value $\infty$.

There is a close analogue of this concept for measures. Suppose that $\mu$ is a (Borel regular outer-) measure on $X$. For simplicity, we always assume that our measures are fully supported, that is $\spt\mu=X$. The \emph{Assouad dimension} of $\mu$, denoted by $\dim_A\mu$ is the infimum of those $\alpha\ge 0$ such that for all $x\in X$, $0<r<R<\infty$, 
\begin{equation}\label{Assouad_def}
\frac{\mu(B(x,R))}{\mu(B(x,r))}\le C\left(\frac{R}{r}\right)^\alpha\,,
\end{equation}
where again $C<\infty$ is a constant depending only on $\alpha$. If \eqref{Assouad_def} is not fulfilled for any $\alpha<\infty$, we define $\dim_A\mu=\infty$. 
This concept has appeared in the literature also under the names upper regularity dimension
and upper Assouad dimension
and measures satisfying the condition \eqref{Assouad_def} have been called $\alpha$-homogeneous.

It is well known and easy to see that $\dim_A X<\infty$ if and only if $X$ is geometrically doubling in the following sense: There is a constant $C<\infty$ such that 
\[N(x,2r,r)\le C\] 
for all $x\in X$, $0<r<\infty$. Likewise,  $\dim_A\mu<\infty$ if and only if 
\[
\frac{\mu\bigl( B(x,2r) \bigr)}{\mu\bigl( B(x,r) \bigr)} \leq C\,,
\]
for some constant $C<\infty$, and for all $x \in X$ and $0<r<\infty$. 

The investigation of the relation between the Assouad dimension of sets and that of measures supported on them was pioneered by Vol{\cprime}berg and Konyagin \cite{VK1984, VK1987}, who proved that for a compact metric space $X$ with finite Assouad dimension,
\begin{equation}\label{eq:VK}
\dim_A X=\inf\{\dim_A\mu\,:\,\spt\mu\subset X\}\,.
\end{equation}
Note that $\dim_A X$ is an obvious lower bound for $\dim_A\mu$, whenever $\mu$ is supported on $X$.
Luukkainen and Saksman \cite{LS1998} later extended \eqref{eq:VK} for complete metric spaces. These results, however, left open the question whether all values $s>\dim_A X$ are attainable for $\dim_A\mu$. In other words, is it possible that the Assouad dimension of measures omits some values $s>\dim_A X$? Plausibly not, but an example of Hare, Mendivil, and Zuberman \cite[Proposition 3.8]{HMZ2020} of an infinite compact set $X\subset\R$ such that 
\[\{\dim_A\mu\,:\,\spt\mu\subset X\}=\{0,\infty\}\] 
suggests that this could be a subtle issue. An analogous intermediate value problem for the Assouad dimension of subsets has been considered in \cite{WWW2016,Chenetal2020}. 

Quite recently, Hare, Mendivil, and Zuberman \cite{HMZ2020}, have shown that if $X\subset\R$ is compact and $\dim_A X>0$, then $\dim_A\mu$ indeed attains all values $s>\dim_A X$ when $\mu$ ranges over all measures supported on $X$. In this note, we prove
\begin{theorem}\label{thm:main}
	If $X$ is a complete metric space with $0<\dim_A X<\infty$, then for all $s>\dim_A X$, there is a measure $\mu$ on $X$ such that $\dim_A\mu=s$. 
\end{theorem}
Theorem \ref{thm:main} extends the result of Hare, Mendivil, and Zuberman to arbitrary complete metric spaces and thus settles the problem raised in \cite[Remark 4.3]{HMZ2020}. We  stress that there is no imposed upper bound on $R$ in the definitions \eqref{eq:Assouad_set_def} and \eqref{Assouad_def} and our result thus makes perfect sense also for unbounded $X$.

Our starting point to prove Theorem \ref{thm:main} is similar to that in \cite{HMZ2020} as we both rely on the construction of generalized cubes in \cite{KRS2012}, see Theorem \ref{thm:dyadic} below. In \cite{HMZ2020}, the authors make use of the geometry of the real-line and impose additional regularity assumptions on the generalized cubes. Our method works in the generality of \cite{KRS2012} without any additional technical  conditions allowing us to obtain the result in its full generality. Another difference is that while for  $s>\dim_A X$, \cite{HMZ2020} provides an ad hoc construction yielding a measure with $\dim_A\mu=s$, we consider a natural family of measures $\mu_p$ and relying on the continuity of the map $p\mapsto\dim_A\mu_p$ verify that $\dim_A\mu_p$ attains all values in $]\dim_A X,\infty[$. 

A result analogous to Theorem \ref{thm:main} for the lower dimension follows by essentially the same proof. We discuss this issue briefly in the last section of the paper.

\section{Toolbox: Generalized cubes}

We begin by restating \cite[Theorem 2.1]{KRS2012}, which is our main tool. It provides us an analogue of the $M$-adic cubes in Euclidean spaces. These ``generalized nested cubes'' will be used throughout these notes.

\begin{theorem}\label{thm:dyadic}
  Let $X\neq\varnothing$ be a metric space such that $N(x,2r,r)<\infty$ for all $x\in X$, $r>0$. For each $0<\delta<\tfrac17$, there
  exists at most countable collections $\mathcal{Q}_k$, $k\in\Z$, of
  Borel sets having the following properties:
  \begin{enumerate}[(i)]
    \item\label{i} $X = \bigcup_{Q\in\mathcal{Q}_k}Q$ for every $k \in \mathbb{Z}$,
    \item\label{ii} 
     If $Q\in\mathcal{Q}_k$, $Q'\in\mathcal{Q}_{k'}$, $k'\ge k$ and $Q\cap Q'\neq\varnothing$, then $Q'\subset Q$.
    \item\label{iii} For every $Q\in\mathcal{Q}_k$, there is a distinguished point $x_{Q} \in X$  so that
    \[
      B\left(x_{Q},\frac13\delta^k\right) \subset Q \subset B\left(x_{Q},2\delta^k\right)\,.
    \]
     \item\label{iv} There exists a point $x_0 \in \bigcap_{k\in\Z}\{x_Q\,:\,Q\in\mathcal{Q}_k\}$. 
\item\label{v} $\{x_{Q}\,:\,Q\in\mathcal{Q}_k\}\subset\{x_{Q}\,:\,Q\in
  \mathcal{Q}_{k+1}\}$ for all $k\in\Z$.
  \end{enumerate}
\end{theorem}

We make some remarks and introduce some further notation. 
We first note that the condition $0<\delta<\tfrac17$  is purely technical and the result holds also for any $\tfrac17\le \delta<1$ with a cost of reducing the parameters $\tfrac13$ and $2$ in \emph{(iii)}. 
We also note that, even if $\delta$ is fixed, the families $\mathcal{Q}_k$ are in no way unique. However, for the sake of simplicity, when we refer to Theorem \ref{thm:dyadic} with a given value of the parmaeter $\delta$, we always assume the families $\mathcal{Q}_k$ have been fixed. We will call the elements of $\bigcup_{k\in\Z}\mathcal{Q}_k$ \emph{cubes} and  refer to the point $x_Q$ as the \emph{center} of the cube $Q$. The distinguished point $x_0$ may be considered the ``origin'' of the space $X$. We also denote by $Q_0$ the unique element of $\mathcal{Q}_0$ that contains $x_0$. This is the ``unit cube'' of $X$.

 	If $Q\in\mathcal{Q}_{k}$, $Q'\in\mathcal{Q}_{k+m}$, $m\in\N$, and $Q'\subset Q$, we denote $Q'\prec_m Q$. We also abbreviate $\prec_1$ to $\prec$. If $Q'\prec Q$, we say that
 	$Q$ is the \emph{parent} of $Q'$, and that $Q'$ is a \emph{child} of $Q$. More generally, we will refer to cubes $Q'\prec_m Q$, $m\in\N$, as \emph{offspring} of $Q$.
 	It is convenient to think about the child-cube $Q'\prec Q$ with $x_{Q'}=x_{Q}$ as a distinguished \emph{central subcube} of $Q$. If $Q'\prec Q$ is not a central subcube, we say that $Q$ is a \emph{boundary cube}. 
 	Let 
 	\[M_Q=|\{Q'\,:\,Q'\prec Q\}|\]
 	stand for the number of children of $Q$. If $\dim_A X<\infty$, then $M_Q$ is bounded and we denote \[M:=\max_{n\in\Z\,,\,Q\in\mathcal{Q}_n}M_Q\,.\] 
 Indeed, it is easy to check that 
 \begin{equation}\label{eq:M_bound}
 M\le K\delta^{-\dim_A X}\,,
 \end{equation}
 where $K=K(\delta)$ is such that $\log_\delta K(\delta)\longrightarrow 0$ as $\delta\downarrow 0$.

 	To demonstrate Theorem \ref{thm:dyadic} and our later considerations, it is helpful to keep in mind the following simple example.
 	\begin{example}\label{triadic_ex1}
       Let $\delta=\tfrac13$, $X=\R$ and let $\mathcal{Q}_k$ consist of the triadic half-open intervals
    \[\mathcal{Q}_k=\left\{\left[(j-\tfrac12)3^{-k},(j+\tfrac12)3^{-k}\right[\,:\,j\in\Z\right\}\,.\]	
 Now $x_Q$ is literally the center point of each interval $Q$ and we may, for instance, fix $x_0=0$ so that $Q_0=[-\tfrac12,\tfrac12[$. Moreover, $M_Q=3$ for all $Q\in\mathcal{Q}_k$, $k\in\Z$ and thus $M=3=\delta^{-\dim_A X}$. 	
 \end{example}

Before proceeding further, we introduce our last bit of notation. For each $t>0$, we denote by $n_t$ the smallest integer such that $\delta^{n_t}\le t$, that is,
\[n_t=\ulcorner \log t/\log \delta\urcorner\,.\]
This notation allows us to switch to a logarithmic scale in the ratio $R/r$. Namely, if $N=n_r-n_R$, then
\begin{equation}\label{eq:log_scale}
\frac{1}{C}\delta^{-N}\le \frac{R}{r}\le C\delta^{-N}\,.
\end{equation}
Here, and in what follows, we denote by $C$ a positive and finite constant that only depends on $\delta$ and $M$ and whose precise value is of no importance.

 	Our next lemma provides us doubling measures that respect the construction of the generalized cubes. This is a quantitative version of the Vol{\cprime}berg-Konyagin Theorem \eqref{eq:VK} in our situation.

\begin{lemma}\label{mulemma}
Suppose that $X\neq \varnothing$ is a complete metric space. Let $0<\delta<\tfrac17$ and  $0<p<\tfrac1M$.  
 Then there is a measure $\mu_{p}$ on $X$ such that $\dim_A\mu_p<\infty$ and so that for all $Q'\prec Q$ it holds that 
 	\begin{equation}\label{eq:mass_distr_p}
 	\mu_p(Q')=\begin{cases}
 		p\mu_p(Q), &\text{if }Q'\text{ is a boundary cube,}\\
 		\left(1-(M_{Q}-1)p\right)\mu_p(Q), &\text{if }Q'\text{ is the central subcube of }Q\,.
 	\end{cases}
 	\end{equation}
Moreover, if $t>\dim_A X$, then $\dim_A\mu_{p}<t$ for some choice of $\delta$ and $p$.
\end{lemma}

\begin{proof}
The claim is implicitely derived in the proof of \cite[Theorem 3.1]{KRS2012} (see Remark 5.1 (2) in the same paper) but since in \cite{KRS2012} it was stated in a less quantitative form, let us briefly sketch the main idea. 

The construction of the measures $\mu_p$ satisfying \eqref{eq:mass_distr_p}, given the families $\mathcal{Q}_k$, is detailed in \cite{KRS2012} and we thus take their existence for granted. It is easy to check that for the measures $\mu_p$, the Assouad dimension may be determined by looking at the ratios of the measures of cubes along offspring chains. In particular, if $\alpha>0$ is such that
\begin{equation}\label{eq:dimAmup}
\frac{\mu(Q)}{\mu(Q')}\le\delta^{-m\alpha}\,,
\end{equation}
whenever $m\in\N$, $Q,Q'\in\bigcup_k\mathcal{Q}_k$, and $Q'\prec_m Q$, then we may conclude that $\dim_A\mu_p\le\alpha$. This is a consequence of the fact that if $x\in X$, $t>0$, and $x\in Q\in\mathcal{Q}_{n_t}$, then
\begin{equation}\label{eq:ballsandcubes}
\frac{1}{C}\mu_p(Q)\le \mu_p\left(B(x,t)\right)\le C\mu_p(Q)\,.
\end{equation}
To verify \eqref{eq:ballsandcubes} is an easy exercise using \eqref{eq:mass_distr_p} and Theorem \ref{thm:dyadic}, or see the proof of \cite[Theorem 3.1]{KRS2012}.
Thus, the task is to show that for suitably chosen $\delta$ and $p$, the upper bound \eqref{eq:dimAmup} holds for $\alpha$ close to $\dim_A X$.

To that end, let us fix $0<\delta<\tfrac17$ and the families $\mathcal{Q}_k$ provided by Theorem \ref{thm:dyadic}. Consider $0<p<\tfrac1M$ and let $\mu_p$ satisfy \eqref{eq:mass_distr_p}. Let $Q,Q'\in\bigcup_{k}\mathcal{Q}_k$ and suppose that $Q'\prec_m Q$. Using \eqref{eq:M_bound}, we have
\begin{equation}\label{eq:p_del}
\begin{split}
\frac{\mu_p(Q)}{\mu_p(Q')}&\le p^{-m}
=(pM)^{-m} M^m\\
&\le(pM)^{-m} K^m\delta^{-m\dim_A X}\\
&=\delta^{-m(\dim_A X-\log_{\delta} K+\log_\delta(pM))} 
\end{split}
\end{equation}
Thus, \eqref{eq:dimAmup} holds for $\alpha=\dim_A X-\log_{\delta} K+\log_\delta(pM)$ and whence
\[\dim_A\mu_p\le \dim_A X-\log_{\delta} K+\log_\delta(pM)\,.\] 
The upper bound can be made arbitrarily close to $\dim_A X$ by first choosing $\delta$ small enough and then $p$ close enough to $\tfrac1M$.
\end{proof}

\begin{remark}
(1) The condition \eqref{eq:mass_distr_p} implies that (having the nested cubes construction fixed), the measures $\mu_{p}$ are uniquely defined up to a multiplicative constant. If we further require that $\mu_{p}(Q_0)=1$, the measure $\mu_{p}$ is thus completely determined by this condition.

(2) It is instructive to think about the measures $\mu_p$ in the setting of the Example \ref{triadic_ex1}, where they are uniquely defined for all $0<p<\tfrac12$ via \eqref{eq:mass_distr_p} and the condition $\mu_p[-\tfrac12,\tfrac12[=1$. In this case, the measure of each triadic interval is split among its triadic child-intervals in such a way that the central child inherits $(1-2p)$ times the mass of its parent and the rest is split equally between the two boundary children.  
It is an easy exercise to verify that  $\dim_A\mu_p=-\log_3 p$, for all $0<p\le\tfrac13$. Thus, $\dim_A\mu_p$ varies continuously in $p$ and attains all values $\ge 1=\dim_A\R$.  Our main result, Theorem \ref{thm:main}, and its proof, reflect this phenomenon in the more abstract setting of Lemma \ref{mulemma}.  
\end{remark}

To conclude this section, we provide the following variant of \cite[Lemma 3.5]{HMZ2020}.

\begin{lemma}\label{lem:dimA>0}
	Suppose that $\dim_A X>0$. Then,  
	for a suitably small $\beta>0$ (depending only on $\dim_A X$, $\delta$, and $M$), there are arbitrarily large $N\in\N$ and cubes 
	\begin{equation}\label{eq:cube_seq}
	Q^N\prec Q^{N-1}\prec\ldots\prec Q^1\prec Q^0
	\end{equation}
	such that at least $\beta N$ of the cubes $Q^1,\ldots,Q^N$ are boundary cubes. 
\end{lemma}

\begin{proof}
	Suppose that $\beta>0$ does not satisfy the requirement of the lemma. We prove the lemma by deriving a lower bound for $\beta$. To that end, we fix $N_0$ so large that any cube sequence as in \eqref{eq:cube_seq} with $N\ge N_0$ contains at most $\beta N$ boundary cubes. Then, for all $k\in\Z$, all $Q\in\mathcal{Q}_k$, and all $N\ge N_0$, it holds that
	\begin{equation}\label{binom}
	\left|\left\{Q'\,:\,Q'\prec_N Q\right\}\right|\le\binom{N}{\lfloor\beta N\rfloor} M^{\beta N}\,.
	\end{equation}	
	To see this, label the offspring of $Q$ as follows. Denote by $Q_1$ the central child-cube of $Q$ and by $Q_2,\ldots,Q_{M_Q}$ the boundary children. For each $Q_i$, denote by $Q_{i1}$ the central child-cube of $Q_i$ and by $Q_{i2},\ldots, Q_{iM_{Q_i}}$ the boundary children. Continuing this labelling for $N$ generations of offspring of $Q$, each $Q'\prec_N Q$ gets labelled with a sequence $\iii\in\{1,\ldots,M\}^N$. 
	But, for each $Q_\iii\prec_N Q$, $\iii=i_1\ldots i_N$, only at most $\beta N$ of the symbols $i_k$
	differ from $1$, and for each such $i_k$, there are at most $M$ admissible values. Whence the bound \eqref{binom}.
	
	Let $x\in X$, $0<r<R<\infty$ and consider a cube  $Q\in\mathcal{Q}_{n_R}$ with $Q\cap B(x,R)\neq\varnothing$. There are at most $C$ such cubes. Switching to the logarithmic scale $N=n_r-n_R$ and recalling \eqref{binom}, we observe that the ball $B(x,R)$ may be covered by 
	\begin{equation}\label{binom_eps}
	C\binom{N}{\lfloor\beta N\rfloor} M^{\beta N}
	\end{equation}
	cubes $Q'\in\mathcal{Q}_{n_r}$. 
	Using trivial bounds for factorials, e.g.
	\[n^ne^{1-n}\le n!\le(n+1)^{n+1}e^{-n}\,,\]	and taking logarithms in base $\delta$, we note that
	\[\binom{N}{\lfloor\beta N\rfloor}\le C\delta^{-\kappa(\beta)N}\,,\]
	where $\kappa(\beta)\to 0$ as $\beta\to 0$.
	Thus, \eqref{binom_eps} is bounded from above by
	\begin{align*}
	C \delta^{-N(\kappa(\beta)-\beta\log_\delta M)}\,.
	\end{align*}
	Since each $Q'\in\mathcal{Q}_{n_r}$ may be covered by $C$ balls of radius $r$, and because $\tfrac{R}{r}\ge C\delta^{-N}$, recall \eqref{eq:log_scale}, it follows that
	\[N(x,R,r)\le C\left(\frac{R}{r}\right)^{\kappa(\beta)-\beta\log_\delta M}.\]
	This upper bound is valid irrespective of $x$, $r$ and $R$ and thus 
	\[\dim_A X\le\kappa(\beta)-\beta\log_\delta M\,,\] 
	yielding the required lower bound for $\beta$.
\end{proof}

\section{Proof of Theorem \ref{thm:main}}

In this section, we prove the following two lemmas. Theorem \ref{thm:main} follows readily from these lemmas and the last claim of Lemma \ref{mulemma} along with the intermediate value theorem for continuous functions. We fix $0<\delta<\tfrac17$ and consider the families $\mathcal{Q}_k$ provided by Theorem \ref{thm:dyadic} along with the measures $\mu_p$ provided by Lemma \ref{mulemma}.

\begin{lemma}\label{mulim}
	If $\dim_A X>0$, then $\lim_{p\downarrow 0}\dim_A\mu_{p}=\infty$.
\end{lemma}

\begin{lemma}\label{mucont}
The map $p\mapsto\dim_A\mu_{p}$ is continuous on the interval $]0,\tfrac1M[$.
\end{lemma}

\begin{proof}[Proof of Lemma \ref{mulim}]
	We estimate $\dim_A\mu_p$ from below. Let $\beta>0$ be the constant from Lemma \ref{lem:dimA>0}. Then, there are arbitrarily long offspring chains 
	\[Q^N\prec Q^{N-1}\prec\ldots\prec Q^1\prec Q^0\] 
	with at least $\beta N$ of the cubes $Q^1,\ldots,Q^N$ boundary cubes. Whence, using \eqref{eq:mass_distr_p},
	\[\mu_p(Q^N)\le p^{\beta N}\mu_p(Q^0)\,.\]
	Let  $x=x_{Q^N}$ and pick $n\in\Z$ such that $Q^0\in\mathcal{Q}_n$. Then, for $r=\tfrac1C\delta^{n+N}$, $R=C\delta^{n}$, we have 
	$B(x,r)\subset Q^N$ and $Q^0\subset B(x,R)$. Thus
	\begin{align*}
	\frac{\mu(B(x,R))}{\mu(B(x,r))}\ge\frac{\mu(Q^0)}{\mu(Q^N)}\ge p^{-\beta N}=\delta^{-N(\beta\log_\delta p)}\ge C\left(\frac{R}{r}\right)^{\beta\log_\delta p}\,.
	\end{align*}
	Since $N$, and thus the ratio $R/r$, may be arbitrarily large, this shows that 
	\[\dim_A\mu_p\ge\beta\log_\delta p\longrightarrow\infty\,,\]
	as $p\downarrow 0$.
\end{proof}	

\begin{proof}[Proof of Lemma \ref{mucont}]
Recall from \eqref{eq:ballsandcubes} that for all $x\in X$, $t>0$, and $x\in Q\in\mathcal{Q}_{n_t}$, we have
\begin{equation}\label{eq:red_to_cubes}
\frac{1}{C}\mu_p(Q)\le \mu_p\left(B(x,t)\right)\le C\mu_p(Q)\,.
\end{equation}
More precisely, the constant $C$ may be taken to be uniform in $p$, if $p$ is bounded away from $0$, which we can of course assume.

Consider $x\in X$ and $0<r<R<\infty$. We switch to the logarithmic scale $N=n_r-n_R$.
Let $Q_R\in\mathcal{Q}_{n_R}$, $Q_r\in\mathcal{Q}_{n_r}$ be the unique cubes containing $x$. Then $Q_r\prec_N Q_R$ and by virtue of \eqref{eq:red_to_cubes},
\begin{equation}\label{B_rtoQ_r}
\frac{1}{C}\,\frac{\mu_p(Q_R)}{\mu_p(Q_r)}\le \frac{\mu_p(B(x,R))}{\mu_p(B(x,r))}\le C\,\frac{\mu_p(Q_R)}{\mu_p(Q_r)}\,.
\end{equation}

Consider the parental line of cubes from $Q_R$ to $Q_r$ by denoting 
$Q_r=Q^N\prec Q^{N-1}\prec\ldots\prec Q^1\prec Q^0=Q_R$. 
Let $\mathfrak{B}=\{i\in\{1,\ldots,N\}\,:\,Q^i\text{ is a boundary cube}\}$ and $\mathfrak{C}=\{1,\ldots,N\}\setminus\mathfrak{B}$.
Using \eqref{eq:mass_distr_p}, it follows that
\[\frac{\mu_p(Q_r)}{\mu_p(Q_R)}=p^{|\mathfrak{B}|}\prod_{n\in\mathfrak{C}}\left(1+(1-M_{Q^{n-1}})p\right)\,.\]

Thus, given $0<p<p'<\frac1M$, we have
\[\left(\frac{1-(M-1)p'}{1-(M-1)p}\right)^N\le \frac{\mu_p(Q_R)/\mu_p(Q_r)}{\mu_{p'}(Q_R)/\mu_{p'}(Q_r)}\le\left(\frac{p'}{p}\right)^N\,.\]
Recalling \eqref{eq:log_scale} and \eqref{B_rtoQ_r} and expressing these upper and lower bounds as
\begin{align*}
\left(\frac{p'}{p}\right)^N&=\delta^{N(\log_\delta p'-\log_\delta p)}\,,\\
\left(\frac{1-(M-1)p'}{1-(M-1)p}\right)^N&=\delta^{N\left(\log_\delta\left(1-(M-1)p'\right)-\log_\delta\left(1-(M-1)p\right)\right)}\,,
\end{align*}
we note that if one of the measures $\mu_p$, $\mu_p'$ satisfies the condition \eqref{Assouad_def} with the exponent $\alpha$, then the other satisfies it with the exponent $\alpha+\varepsilon$, where
\begin{equation}\label{eq:key_2}
\varepsilon=\max\left\{\log_\delta p-\log_\delta p', \log_\delta\left(1-(M-1)p'\right)-\log_\delta\left(1-(M-1)p\right)\right\}\,.
\end{equation}
Thus, we observe a quantitative modulus of continuity for $p\mapsto\dim_A\mu_p$.
\end{proof}

\begin{remark}
The key estimate derived in the proof of Lemma \ref{mucont} is the following: Given $0<p<\tfrac1M$, and $\varepsilon>0$, we have 
\begin{equation}\label{eq:key_estimate}
\delta^{\varepsilon N}\le \frac{\mu_p(Q')/\mu_p(Q)}{\mu_{p'}(Q')/\mu_{p'}(Q)}\le\delta^{-\varepsilon N}\quad\quad\text{for all }\quad\quad Q\prec_N Q'\,,
\end{equation}
provided $|p'-p|$ is small enough, see \eqref{eq:key_2}. This estimate actually implies the continuity of $p\mapsto \dim \mu_p$ for a variety of Assouad type dimensions $\dim$ in addition to $\dim_A$ and the lower dimension $\dim_L$ (considered in Section \ref{sec:lower} below). For some variants of the Assouad dimension such as the upper and lower Assouad spectrums  introduced in \cite{HT2019}, and the related quasi-Assouad dimensions, the continuity is obvious from \eqref{eq:key_estimate}. Moreover, if $X$ is compact, then it is relatively easy to derive the continuity of 
$p\mapsto\dim_q\mu_p$,  $p\mapsto\dim_M\mu_p$, $p\mapsto\dim_F\mu_p$ using \eqref{eq:key_estimate}. Here 
$\dim_q\mu$, where $1\neq q>0$, denotes the $L^q$-dimension, see e.g. \cite{KRS2016} for the definition. Moreover, $\dim_M\mu$ and $\dim_F\mu$ denote the Minkowski and  Frostman dimensions of $\mu$, respectively, as defined in \cite{FK2020}.
\end{remark}

\section{Attainable values for the lower dimension}\label{sec:lower}

In this final section, we discuss the following ``dual'' result for Theorem \ref{thm:main}. 
\begin{theorem}\label{thm:lower}
If $X$ is a complete metric space with $\dim_A X<\infty$, then for all $0<s<\dim_L X$, there is a measure $\mu$ with support $X$ such that $\dim_L\mu=s$.
\end{theorem} 
Here
$\dim_L X$ is  the lower dimension of $X$ defined as the supremum of the exponents $\alpha$,  for which there is a constant $0<C<\infty$ such that 
\[N(x,R,r)\ge C\left(\frac{R}{r}\right)^\alpha\,,\]
 for all $x\in X$ and $0<r<R<\diam(X)$. Analogously, $\dim_L\mu$ is the supremum of those $\alpha$, for which
\[\frac{\mu(B(x,R))}{\mu(B(x,r))}\ge C\left(\frac{R}{r}\right)^\alpha\]
holds irrespective of $x\in X$ and $0<r<R<\diam(X)$. 

For a compact set of reals, this result is also due to Hare, Mendivil, and Zuberman, see \cite[Theorem 4.1]{HMZ2020}. In its full generality, Theorem \ref{thm:lower} may be proved using ideas very similar to our proof of Theorem \ref{thm:main}. Thus, we only sketch the idea and leave the details for the interested reader.

To begin with, we observe that
$p\mapsto\dim_L\mu_p$ is continuous. This follows by the very same proof as Lemma \ref{mucont}, see \eqref{eq:key_estimate}. Likewise,  switching to a chain of central cubes in the proof of Lemma \ref{mulim} gives $\lim_{p\downarrow 0}\dim_L\mu_p=0$. So if we knew that for $s<\dim_L X$, we had $\dim_L\mu_{p}>s$ for a suitably chosen $\delta$ and $p$, the claim would follow just as in the case of Theorem \ref{thm:main}. This holds in some nice cases such as the Example \ref{triadic_ex1}, but unfortunately it does not hold in general if there is variation in the numbers $M_Q$. 

However, the argument can still be saved by considering mass distributions more general than those defined by \eqref{eq:mass_distr_p}.
Given $J\in\N$, let $0<p<\tfrac1M$ and let $\eta=(\eta_1,\ldots,\eta_J)$ be a probability vector, where all the weights $\eta_i$ are non-zero. For each cube $Q\in \mathcal{Q}_k$  ($\delta^k< C\diam(X)$),
 we consider $J$ distinguished central subcubes $Q(1),\ldots,Q(J)\prec Q$ whose distance to the complement of $Q$ is comparable to $\delta^{k}$. Note that this is possible, provided $\delta$ is small enough depending on $J$ and $\dim_L X$, see e.g. \cite[Proof of Theorem 3.2]{KL2017}. 
Again, each boundary cube $Q'\prec Q$ inherits $p$ times the mass of its parent and the rest is distributed among the central child-cubes $Q(1),\ldots,Q(J)$ in the proportion of the weights $\eta_1,\ldots,\eta_J$. Denote the resulting measure by $\mu_{p,\eta}$. Note that the measures $\mu_p$ considered earlier correspond to the case $J=1$ and the trivial probability vector $\eta=(1)$. 

The proof of Lemma \ref{mucont} still works in this setting and implies that 
\begin{equation}\label{f_cont}
(p,\eta)\mapsto\dim_L\mu_{p,\eta}\text{ is continuous.}
\end{equation}
 Moreover, if e.g. $p\to 0$ and $\eta\to(1,0,\ldots,0)$, then it is very easy to see, by following a cube chain $Q^N\prec Q^{N-1}\ldots\prec Q^1\prec Q^0$, where $Q^{i+1}=Q^i(1)$ for each $i=0,\ldots,N-1$, that 
\begin{equation}\label{f_lim}
\dim_L\mu_{p,\eta}\longrightarrow 0\,.
\end{equation}
Finally, a result of K\"aenm\"aki and Lehrb\"ack \cite{KL2017} implies that given $s<\dim_L X$, it is possible to choose $\delta$, $J$, $p$, and $\eta$ so that 
\begin{equation}\label{f_dim}
\dim_L\mu_{p,\eta}>s\,.
\end{equation} In fact, for \eqref{f_dim}, it is enough to consider the uniform probability vector $\eta=(\tfrac1J,\tfrac1J,\ldots,\tfrac1J)$ for a suitably chosen $\delta>0$, $J\in\N$, and $p$, see \cite[Theorem 3.4]{KL2017}.  Theorem \ref{thm:lower} now follows by combining \eqref{f_cont}--\eqref{f_dim} and the intermediate value theorem for continuous functions.

\bibliographystyle{plain}
\bibliography{assouad_refs}

\end{document}